\documentclass[12pt]{amsart}
\usepackage[all]{xy}
\usepackage{amsfonts}
\usepackage{amssymb}
\usepackage{pb-diagram}
\usepackage{amsmath}
\usepackage{color}

\newtheorem{teo}{Theorem}[section]
\newtheorem{lem}[teo]{Lemma}
\newtheorem{prop}[teo]{Proposition}

\newtheorem{dfn}[teo]{Definition}
\newtheorem{rk}[teo]{Remark}
\newtheorem{ex}[teo]{Example}

\begin{document}

\title[Strict essential extensions of $C^*$-algebras and
 Hilbert $C^*$-modules]
{Strict essential extensions of $C^*$-algebras and
 Hilbert $C^*$-modules}

\author{M. Frank}
\thanks{}
\address{Hochschule f\"ur Technik, Wirtschaft und Kultur (HTWK)
Leipzig, Fachbereich IMN, PF 301166, D-04251 Leipzig, F.R.
Germany} \email{mfrank@imn.htwk-leipzig.de} \urladdr{}

\author{A. A. Pavlov}
\thanks{Partially supported by the RFBR (grant 07-01-91555 and
grant 07-01-00046)}
\address{Moscow State University, Department of Geography and
Department of Mechanics and Mathematics, 119 992 Moscow, Russia}
\email{axpavlov@mail.ru} \urladdr{
http://www.freewebs.com/axpavlov}

\begin{abstract} In the present paper we  develop both ideas
of~\cite{Bakic} and the categorical approach to multipliers
from~\cite{Lance, Pav1, Pav2} for the introduction and study of
left multipliers of Hilbert $C^*$-modules. Some properties and, in
particular, the property of maximality among all strictly
essential extensions of a Hilbert $C^*$-module for left
multipliers are proved. Also relations between left essential and
left strictly essential extensions in different contexts are
obtained. Left essential and left strictly essential extensions of
matrix algebras are considered. In the final paragraph the
topological approach to the left multiplier theory of Hilbert
$C^*$-modules is worked out.
\end{abstract}

\maketitle
\section{Introduction}
There are a lot of ways to include a non-unital $C^*$-algebra as
an (essential) ideal into unital ones. Among those extensions
there is a maximal extension, the algebra of multipliers. This
object can be considered from different points of view.
Historically the first definition of this algebra arose in the
context of centralizers in~\cite{Busby}. There exists another
definition of multipliers given via the universal representations
of $C^*$-algebras. This approach may be generalized for arbitrary
non-degenerated faithful representations of $C^*$-algebras on
Hilbert $C^*$-modules, cf.~\cite{Lance}. Besides this we can
understand algebras of multipliers as the set of all adjointable
maps from a $C^*$-algebra to itself. Indeed, the latter approach
is the most suitable for a generalization of these constructions
to Hilbert $C^*$-modules. In~\cite{Bakic} multipliers of Hilbert
$C^*$-modules were introduced and their universal property was
obtained. In~\cite{Raeburn} these notions were significantly used
both for a generalization of the Kasparov stabilization theorem to
the non-unital case and for an extension of the concept of module
frames in Hilbert $C^*$-modules over non-unital $C^*$-algebras as
a continuation of ideas on module frame concepts explained
in~\cite{Frank1,Frank2}.

In the present paper we will continue both the ideas
of~\cite{Bakic} and the categorical approach to multipliers
from~\cite{Lance, Pav1, Pav2} for an introduction of a notion of
left multipliers of Hilbert $C^*$-modules. The text is organized
in the following way. To $\S 2$ we include some remarks on Hilbert
$C^*$-modules and on categorical constructions of (left)
multipliers of $C^*$-algebras. In $\S 3$ some properties and, in
particular, the property to be a left strictly essential extension
(Theorem~\ref{teo:mod_l_s_es_ext}) and the property of maximality
(Theorem~\ref{teo:left_mult_mod}) for left multipliers of Hilbert
$C^*$-modules are obtained. In $\S 4$ we study differences between
essential and strictly essential extensions both in $C^*$- and
Banach situations. In $\S 5$ left essential and left strictly
essential extensions of matrix algebras are studied and the
property of their maximality are considered. Finally, $\S 6$ is
dedicated to the approach to the left multiplier theory
considering appropriate analogs of strict topologies.

\section{Preliminaries and reminding}
To begin with, let us remind that for a $C^*$-algebra $A$ a
\emph{pre-Hilbert $A$-module } is a (right) $A$-module $V$
equipped with a semi-linear map $\langle\cdot,\cdot\rangle
:V\times V\rightarrow A$ such that
\begin{enumerate}
    \item $\langle x,x\rangle\ge 0$ for all $x\in V$,
    \item $\langle x,x\rangle= 0$ if and only if $x=0$,
    \item $\langle x,y\rangle^*=\langle y,x\rangle$ for all $x,y
    \in V$,
    \item  $\langle x,ya\rangle=\langle x,y\rangle a$ for all
    $x,y\in V, a\in A$.
\end{enumerate}
The map $\langle\cdot,\cdot\rangle$ is called an $A$-valued inner
product. A norm can be defined for any pre-Hilbert module $V$ by
the formula
$$
\|x\|=\|\langle x,x\rangle\|^{1/2},\quad x\in V.
$$
A pre-Hilbert $A$-module is a \emph{Hilbert $A$-module} if it is
complete with respect to this norm.

Let $V_1$, $V_2$ be Hilbert $A$-modules. Then by ${\rm
Hom}_A(V_1,V_2)$ we will denote the set of all
$A$-linear bounded operators from $V_1$ to $V_2$. When $V_1=V_2=V$
we will write ${\rm End}_A(V)$ instead of ${\rm Hom}_A(V,V)$.

Let us remind that an operator $T\in {\rm Hom}_A(V_1,V_2)$ admits
an adjoint operator if there exists an element $T^*\in {\rm Hom}_A
(V_2,V_1)$ such that
\begin{gather*}
    \langle Tx,y\rangle=\langle x, T^*y\rangle  \quad\text{for
    all}\quad x\in V_1, y\in V_2.
\end{gather*}
By ${\rm End}^*_A(V)$ we denote the subset of ${\rm End}_A(V)$
consisting of operators which possess adjoint ones.

Any $C^*$-algebra $A$ can be considered as a
right Hilbert $A$-module over itself with the inner product
$\langle a,b\rangle=a^*b$. Then the C*-algebra $M(A)$ of
multipliers of $A$ can be defined as $M(A)={\rm End}^*_A(A)$ and
the Banach algebra $LM(A)$ of left multipliers of $A$ can be
defined as $LM(A)={\rm End}_A(A)$,
cf.~\cite{Green,Kasparov,LinPJM}.

In~\cite{Pav1,Pav2} this definition has been
extended to rather more general situations. Let us briefly remind
these notions and results, because we will need them later. Let
$\mathcal{B}$ be a Banach algebra and suppose the existence of a
$C^*$-subalgebra $A \subseteq \mathcal{B}$ which is a left ideal
of $\mathcal{B}$.

\begin{dfn}\rm\label{dfn:left_es_id}
$A$ is said to be a \emph{left essential ideal} of $\mathcal{B}$
if one of the following equivalent conditions holds:
\begin{enumerate}
    \item any two-sided ideal of $\mathcal{B}$ has a non-trivial intersection
    with $A$;
    \item there does not exist any non-zero element $b\in B$ such that
    $ba=0$ for all $a\in A$, i.e.~the two-sided ideal
    $\mathcal{B}_0:=\{b\in \mathcal{B} : ba=0 \quad\text{for all}\quad
    a\in A\}$ of $\mathcal{B}$ equals to zero.
\end{enumerate}
\end{dfn}

\begin{dfn}\rm
 $A$ is said to be a \emph{left strictly essential ideal} of
 $\mathcal{B}$ (and $\mathcal{B}$ is said to be a \emph{left strictly
 essential extension} of $A$)
  if the following condition holds
\begin{gather}\label{eq:str_es_ext_banalg}
    \|b\|=\sup\{\|ba\| : a\in A, \|a\|\le 1\} \quad\text{for
    all}\quad b\in \mathcal{B}.
\end{gather}
\end{dfn}

Any left strictly essential ideal is a left essential one, because
the equality~(\ref{eq:str_es_ext_banalg}) implies the second
condition of Definition~\ref{dfn:left_es_id}. But the inverse
statement is not true, i.e. a left essential ideal of a Banach
algebra might not be a left strictly essential one in
contradistinction to the case when $\mathcal{B}$ is a
$C^*$-algebra, because in the $C^*$-case both these properties of
ideals coincide (see~\cite[Lemmas 7, 8]{Pav1}). Let us remark that
a Banach algebra $\mathcal{B}$ may contain not only the fixed
$C^*$-algebra $A$, but its isomorphic copy too as a left essential
or strictly essential ideal. In such situations we also will call
$\mathcal{B}$ as either a left essential or a left strictly
essential extension of $A$.

Let $A$ be a $C^*$-algebra and $\mathcal{B}$ be an
algebra with an involution containing $A$ as a left essential ideal.
Then the map
\begin{gather*}
    \|\cdot\| : \mathcal{B}\rightarrow [0,\infty)
\end{gather*}
introduced by the formula~(\ref{eq:str_es_ext_banalg}) defines a
norm on $\mathcal{B}$ such that $A$ is a left strictly essential
ideal of $\mathcal{B}$ with respect to it. But, in general,
$(\mathcal{B},\|\cdot\|)$ may not be a Banach algebra because
$\mathcal B$ might not be complete with respect to that norm.

\begin{lem} {\rm (\cite{Pav1}).}\label{lem:cont_to_mult}
Let $A, C$ be $C^*$-algebras, $\mathcal{B}$ be a Banach algebra,
$A\subset \mathcal{B}$ be a left ideal, $E$ be a Hilbert
$C$-module and $\rho : A\rightarrow {\rm End}^*_C(E)$ be a
non-degenerate representation of $A$ in $E$. Then there is a
unique extension of $\rho$ to a morphism $\widetilde{\rho} :
\mathcal{B}\rightarrow {\rm End}_C(E)$ of Banach algebras. If in
addition $A$ is a left strictly essential ideal and $\rho$ is
injective, then $\widetilde{\rho}$ is an isometry.
\end{lem}

\begin{dfn}\rm Let $A, C$ be $C^*$-algebras,
$E$ be a Hilbert $C$-module and $\rho : A\rightarrow {\rm
End}^*_C(E)$ be a faithful non-degenerate representation of $A$ in
$E$. Then $(E,C,\rho)$ is \emph{an admissible for $A$ triple}.
\end{dfn}

\begin{dfn}\rm\label{dfn:left_mult}
Let $(E,C,\rho)$ be an admissible for $A$ triple.
Then the set of left $(E,C,\rho)$-multipliers of $A$ is defined as
\begin{gather*}
    LM_{(E,C,\rho)}(A)=\{T\in {\rm End}_C(E) : T\rho(A)\subset \rho(A)\}
    \, .
\end{gather*}
\end{dfn}

The standard definition of the left multipliers $LM(A)$ of a
$C^*$-algebra $A$ is a special case of Definition~\ref{dfn:left_mult}
corresponding to the triple $(A,A,\alpha)$, where
\begin{gather*}
    \alpha : A\rightarrow {\rm End}^*_A(A),
    \qquad \alpha
    (a)b=ab\quad (a,b\in A)\, .
\end{gather*}

\begin{dfn}\rm\label{dfn:max_s_e_e}
   A left strictly essential extension $\widehat{\mathcal{B}}$ of $A$
   is \emph{maximal} if for any other left strictly essential
   extension $\mathcal{B}$ of $A$ there is an isometrical
   homomorphism from $\mathcal{B}$ to $\widehat{\mathcal{B}}$, which
   acts identically on the two copies of $A$.
\end{dfn}

\begin{teo}{\rm (\cite{Pav1}).}\label{teo:left_mult1}
    For any admissible for $A$ triple $(E,C,\rho)$ the algebra
    of the left $(E,C,\rho)$-multipliers is a maximal left strictly
    essential extension of $A$.
\end{teo}

\begin{teo}{\rm (\cite{Pav1}).}\label{teo:left_mult2}
    The Banach algebras of left $(E,C,\rho)$-multipliers are
    isomorphic for all admissible for $A$ triples $(E,C,\rho)$,
    and these isomorphisms act as the identity map on the embedded
    copies of $A$.
\end{teo}

\section{Left multipliers of Hilbert $C^*$-modules}

In this section we introduce the notion of left multipliers of
Hilbert $C^*$-modules as a particular form of a strict essential
extension of the respective Hilbert $C^*$-module. We prove the
generic maximality for the left multipliers of a Hilbert
$C^*$-module among all of its strict essential extensions.

\begin{dfn}\rm\label{dfn: Ban_ext}
Let $A$ be a $C^*$-algebra, let $V$ be a Hilbert $A$-module.
A \emph{Banach extension} of $V$ is a triple $(W,\mathcal{B},\Phi)$,
where
\begin{enumerate}
    \item $\mathcal{B}$ is a Banach algebra, $A\subset \mathcal{B}$
     is a left ideal;
    \item $W$ is a Banach $\mathcal{B}$-module;
    \item $\Phi : V\rightarrow W$ is an $A$-linear isometric map;
    \item ${\rm Im} \Phi=WA:=\overline{\rm{span}_A}\{y : y\in
    W\}$.
\end{enumerate}
\end{dfn}

\begin{rk}\rm The forth condition of Definition~\ref{dfn: Ban_ext}
is an analogue of the requirement to the representation $\rho$ from
Definition~\ref{dfn:left_mult} to be non-degenerate.
\end{rk}

\begin{dfn}\rm
The Banach extension $(W,\mathcal{B},\Phi)$ of $A$ is
\emph{strictly essential} one if $A\subset \mathcal{B}$ is a left
strictly essential ideal and the following condition holds
\begin{gather}\label{eq:str_es_Ban_ext}
    \|y\|=\sup\{\|ya\| : a\in A,\|a\|\le 1\}\quad\text{for
    all}\quad y\in W.
\end{gather}
\end{dfn}

\begin{ex}\rm\label{ex:id_Ban_ext}
For any Hilbert $A$-module $V$ the triple $(V,A,{\rm id_V})$,
where $\rm {id_V}:V\rightarrow V$ is an identical map, is
(\emph{an identical}) strictly essential Banach extension of $V$,
because for any approximative unit $\{e_\alpha\}$ of $A$ and
for any $x\in V$ the net $xe_\alpha$ converges with respect to
the norm to $x$ (see~\cite[Lemma 1.3.8]{MaTroBook}).
In particular, the triple $(A,A,{\rm id_A})$ is a strictly
essential Banach extension of the Hilbert $A$-module $A$.
\end{ex}

\begin{ex}\rm Let $A\subset \mathcal{B}$ be a left strictly essential ideal
and let us denote this embedding by $i$. Let us consider
$\mathcal{B}$ as a Banach module over itself and $A$ as a Hilbert
module over itself. Then the triple $(\mathcal{B},\mathcal{B},i)$
is a strictly essential extension of
 $A$.
\end{ex}

\begin{dfn}\rm\label{dfn:Ban_s_e_e}
A Banach strictly essential extension
$(\widehat{W},\widehat{\mathcal{B}},\widehat{\Phi})$ of a Hilbert
$A$-module $V$ is \emph{maximal} if for any other Banach strictly
essential extension $(W,\mathcal{B},\Phi)$ there are an
isometrical homomorphism $\lambda : \mathcal{B}\rightarrow
\widehat{\mathcal{B}}$ which is identical on $A$ and an
isometrical linear map $\Lambda : W\rightarrow \widehat{W}$ such
that it is a $\lambda$-homomorphism, i.e.
\begin{gather*}
    \Lambda (yb)=\Lambda(y)\lambda(b)\quad\text{for all}\quad
    y\in W, b\in \mathcal{B},
\end{gather*}
and the following diagram
\[
\begin{diagram}
\node{W}\arrow[2]{e,t}{\Lambda} \node[2]{\widehat{W}}
\\ \node[2]{V}\arrow{ne,r}{\widehat{\Phi}}\arrow{nw,b}{\Phi}
\end{diagram}
\]
is commutative.
\end{dfn}

As a consequence $\Lambda$ maps $\Phi(V)$ onto $\widehat{\Phi}(V)$
identifying both the copies of $V$. If we consider only Banach
algebras $\mathcal{B}$ which are both modules over themselves and
left strict essential extensions of $A$, i.e. only Banach strictly
essential extensions of kind $(\mathcal{B},\mathcal{B},i)$, where
$i$ denotes the embedding of $A$ into $\mathcal{B}$, then apparently
Definition~\ref{dfn:Ban_s_e_e} coincides with
Definition~\ref{dfn:max_s_e_e}.

\begin{dfn}\rm\label{dfn:left_mult_mod}
Let $V$ be a Hilbert module over a $C^*$-algebra
$A$. Then \emph{the set of left multipliers} $LM(V)$ of $V$
is defined by
\begin{gather*}
    LM(V):={\rm Hom_A}(A,V).
\end{gather*}
\end{dfn}

Let us define the action of $LM(A)$ on $LM(V)$ by the formula
\begin{gather}\label{eq:LM_action}
    (yb)(a)=y(b(a)),\qquad y\in LM(V),b\in LM(A),a\in A,
\end{gather}
what turns the set $LM(V)$ into a right $LM(A)$-module.
Furthermore, define the map $\Gamma : V\rightarrow LM(V)$ in the
following way
\begin{gather}\label{eq:def_Gamma}
    (\Gamma (x))(a)=xa, \qquad x\in V, a\in A.
\end{gather}
If $V=A$, then, obviously, Definition~\ref{dfn:left_mult_mod}
coincides with the standard definition of the left multipliers
of the $C^*$-algebra $A$. The map $\Gamma$ is an isometric
embedding of $V$ into $LM(V)$ as a Banach $A$-submodule.

\begin{teo}\label{teo:mod_l_s_es_ext}
   Let $V$ be a Hilbert module over a $C^*$-algebra
   $A$. For $LM(V)$ being the set of left multipliers of $V$ the
   triple $(LM(V),LM(A),\Gamma)$ is a strictly essential Banach
   extension of a Hilbert $A$-module $V$.
\end{teo}

\begin{proof}
It is clear that $LM(V)$ is a Banach $LM(A)$-module with respect
to the action~(\ref{eq:LM_action}). Beside this, $A$ is a left
strictly essential ideal in $LM(A)$ by
Theorem~\ref{teo:left_mult1}. Moreover, it is a straightforward
verification that the map~(\ref{eq:def_Gamma}) is an $A$-module
isometry and that the equality~(\ref{eq:str_es_Ban_ext}) is holds.
So it remains only to check the forth condition of
Definition~\ref{dfn: Ban_ext}.

Let us choose any approximative identity $\{e_\alpha\}$ in $A$.
Then for any $x\in V$ we can write
\begin{gather*}
    \Gamma (x)=\lim_\alpha \Gamma(xe_\alpha)=\lim_\alpha
    \Gamma(x)e_\alpha
\end{gather*}
and, consequently, ${\rm Im}\Gamma\subset LM(V)A$. To obtain the
inverse inclusion let us take any $T\in LM(V)$, $a\in A$, then for
all $b\in A$ we have
\begin{gather*}
    (Ta)(b)=T(ab)=T(a)b=\Gamma(T(a)) b,
\end{gather*}
so $Ta=\Gamma(T(a))$ and we have got the desired set identity
${\rm Im}\Gamma= LM(V)A$.
\end{proof}

\begin{teo}\label{teo:left_mult_mod}
   Let $V$ be a Hilbert module over a $C^*$-algebra $A$. The
   strictly essential Banach extension $(LM(V),LM(A),\Gamma)$
   of any Hilbert $A$-module $V$ is maximal.
\end{teo}

\begin{proof} Let us consider any other strictly
essential Banach extension $(W,\mathcal{B},\Phi)$ of $V$. An
isometrical homomorphism $\lambda : \mathcal{B}\rightarrow LM(A)$
which is identical on $A$ exists by Theorem~\ref{teo:left_mult1},
and, moreover, the uniqueness of this homomorphism
(cf.~Lemma~\ref{lem:cont_to_mult}) is the reason why the equality
$\lambda (b)(a)=ba$ has to hold for all $b\in \mathcal{B}, a\in A$.

Now let us define the map $\Lambda : W\rightarrow LM(V)$ by the
formula
\begin{gather*}
   \Lambda(y)(a):=\Gamma(\Phi^{-1}(y))(a)=
    \Phi^{-1}(ya),\quad y\in W, a\in A \, .
\end{gather*}
This definition is correct because $ya\in {\rm Im}\Phi$ and $\Phi$
is an isometry. Further, for any $y\in W$ the following equalities
hold due to~(\ref{eq:str_es_Ban_ext}):
\begin{gather*}
    \|\Lambda y\|=\sup \{\|\Lambda (y)(a)\| : a\in A,\|a\|\le 1\}=
    \sup\{\|ya\| : a\in A,\|a\|\le 1\}=\|y\|.
\end{gather*}
Consequently, $\Lambda$ is an isometry. The properties of $\Lambda$
to be a $\lambda$-homomorphism and to fulfil the equality $\Lambda
\Phi=\Gamma$ can be derived by obvious computations.
\end{proof}

\section{Essential and strict essential extensions
of Hilbert $C^*$-modules}

The left strict extensions of Hilbert $C^*$-modules are Banach
module extensions over Banach algebras, in general. So a wide
variety of them might occur in particular situations in difference
to the quite canonical situations appearing in the case of
multiplier modules and (two-sided) strict extensions,
cf.~\cite{Bakic}. We start the investigation of characteristic
situations with the known definition of (two-sided) essential
extensions of Hilbert $C^*$-modules for the situation of
$C^*$-extensions of the $C^*$-algebra of coefficients.

\begin{dfn}\rm (\cite{Bakic})
Let $V$ be a Hilbert $A$-module over a $C^*$-algebra $A$.
An \emph{extension} of $V$ is a triple $(W,B,\Phi)$ such that
\begin{enumerate}
    \item $B$ is a $C^*$-algebra, $A\subset B$ is an ideal;
    \item $W$ is a Hilbert $B$-module;
    \item $\Phi : V\rightarrow W$ is a map satisfying
    $\langle \Phi x,\Phi y\rangle=\langle x,y\rangle$
    for all $x,y\in V$;
    \item ${\rm Im} \Phi=WA$.
\end{enumerate}
If in addition $A$ is an essential ideal of $B$, then the
extension $(W,B,\Phi)$ is called \emph{essential}.
\end{dfn}

\begin{teo}
   Let $(W,B,\Phi)$ be an essential extension of a
   Hilbert $A$-module $V$. Then the mentioned extension is
   automatically strictly essential, i.e.
     the conditions {\rm (\ref{eq:str_es_ext_banalg})},
   {\rm (\ref{eq:str_es_Ban_ext})} hold for it.
\end{teo}

\begin{proof}
The following equalities can be established for any $y\in W, a\in
A$:
\begin{gather*}
    \|ya\|^2=\|\langle ya,ya\rangle\|=\|a^*\langle y,y\rangle
    a\|=
    \|\langle y,y\rangle^{1/2}a\|^2.
\end{gather*}
Any essential ideal is automatically strictly essential in the
$C^*$-case , i.e. the property~(\ref{eq:str_es_ext_banalg}) holds
(see~\cite[Lemma 7]{Pav1}). Therefore,
\begin{gather*}
    \sup\{\|\langle y,y\rangle^{1/2}a\| : a\in A, \|a\|\le 1\}=
    \|\langle y,y\rangle^{1/2}\|=\|y\|
\end{gather*}
and, consequently, the desired property is obtained.
\end{proof}

At the contrary, the results in the situation of Banach extensions
of Hilbert $A$-modules is completely different from the one
mentioned above as the next statement shows.

\begin{teo}\label{prop:es_not_str_es_B_e}
For any non-unital C*-algebra $A$ there exists a Hilbert $A$-module
$V$ and a Banach extension $(W,\mathcal{B},\Phi)$ of $V$ such that
$A$ is a strictly essential ideal of $\mathcal{B}$, but the
condition~{\rm (\ref{eq:str_es_Ban_ext})} does not hold.
\end{teo}

\begin{proof} Let us take into consideration any non-unital
$C^*$-algebra $A$ and put $V=A$, $\mathcal{B}=A$. Further let us
choose $W=\widetilde{A}$, where $\widetilde{A}$ is the
$C^*$-algebra with an adjoint unit equipped with the (Banach, but
not $C^*$-) norm
\begin{gather*}
\|(a,\lambda)\|=\|a\|+|\lambda|,\qquad a\in A, \lambda\in
\mathbb{C}.
\end{gather*}
Then $A$ will be a left essential, but not left strictly essential
ideal of $\widetilde{A}$ (see~\cite[Lemma 8]{Pav1}). Let us choose
the map $\Phi = {\it i}$ to be the canonical embedding of $A$ into
$\widetilde{A}$. Then the condition~{\rm
(\ref{eq:str_es_Ban_ext})} does not hold for the Banach extension
$(\widetilde{A}, A, i)$.
\end{proof}

Let $(W,\mathcal{B},\Phi)$ be a Banach extension of a Hilbert
$A$-module $V$. Let us define a closed $\mathcal{B}$-submodule
$W_0$ of $W$ by the formula
\begin{gather}\label{eq:es_ban_mod}
    W_0=\{y\in W : ya=0 \quad\text{for all}\quad a\in A\}.
\end{gather}
Then the assertion $W_0=\{ 0 \}$ would be a reasonable analogue
to the condition (ii) of Definition~\ref{dfn:left_es_id}. Let us
remark in addition that $W_0=\{ 0 \}$ holds for any strict essential
Banach extensions.

\begin{rk}\rm\label{rk:es_not_str_es_B_e}
We can a bit strengthen the result of
Theorem~\ref{prop:es_not_str_es_B_e}. More precisely, the example
$(\widetilde{A},A,i)$ of a Banach extension with a non-unital
$C^*$-algebra $A$ from the proof of Theorem~\ref{prop:es_not_str_es_B_e}
shows that there are a Hilbert $A$-module $V$ and its Banach extension
$(W,\mathcal{B},\Phi)$ such that $A$ is a strictly essential ideal
of $\mathcal{B}$ and $W_0=0$, but the condition~(\ref{eq:str_es_Ban_ext})
does not hold.
\end{rk}

Finally, let us discuss one question, which was formulated by
D.~Baki\'c. In~\cite{Bakic} the set of multipliers of a Hilbert
$A$-module $V$ was defined as the Hilbert $M(A)$-module $V_d=M(V):=
{\rm Hom}^*(A,V)$. Then the question was raised whether it is an
admissible situation, when Hilbert modules $V_1$ and $V_2$ over a
non-unital $C^*$-algebra are not isomorphic, but their modules of
multipliers are isomorphic. Let us demonstrate by example
that the answer on this question is affirmative.

\begin{ex}\rm
Let $A$ be the $C^*$-algebra $K(H)$ of all compact operators on a
separable Hilbert space $H$, and let $B$ be the $C^*$-algebra $B(H)$
of all bounded linear operators on $H$. Consider the $C^*$-algebra
$C$ and two Hilbert $C$-modules $V_1$ and $V_2$ defined by
\[
   C = \left( \begin{array}{cc} A & 0 \\ 0 & B \end{array} \right) \, ,
   V_1 = \left( \begin{array}{cc} A & 0 \\ 0 & B \end{array} \right) \, ,
   V_2 = \left( \begin{array}{cc} A & 0 \\ 0 & A \end{array} \right) \, .
\]
Then $V_1$ and $V_2$ are not isomorphic as Hilbert $C$-modules,
because the first one is a full Hilbert $C$-module, but the second
one is not. At the same time both their sets of (two-sided, left)
multipliers can be described by the Hilbert $M(C)$-module
\[
    M(V_1)=M(V_2)=LM(V_1)=LM(V_2) =
    \left( \begin{array}{cc} B & 0 \\ 0 & B \end{array} \right) \, .
\]
\end{ex}

\section{Essential and strict essential extensions of matrix algebras}

The aim of the present section is to check to which extend an
essential extension of a $C^*$-algebra $A$ to a $*$-algebra
$\mathbb B$ preserves the property of $\mathbb B$ to be Banach
with respect to the induced norm (\ref{eq:str_es_ext_banalg})
for any of its finite matricial extensions $M_n(\mathbb B)$,
$n \ge 2$, with $M_n(A) \subseteq M_n(\mathbb B)$ being an
essential extension of $M_n(A)$, and vice versa.
Note, that Definition \ref{dfn:left_es_id} of a left essential ideal
is formulated for algebraic representations of the $*$-algebra $A$
in $*$-algebras $\mathbb B$ without any reference to topologies on
both these algebras.

\begin{lem}
  Let $A$ be a $C^*$-algebra and $\mathcal{B}$ be a
  Banach algebra. Then the following conditions are equivalent:
  \begin{enumerate}
    \item $A$ is a left essential ideal of $\mathcal{B}$;
    \item $M_n(A)$ is a left essential ideal of $M_n(\mathcal{B})$ for
    any integer $n\ge 2$.
  \end{enumerate}
\end{lem}

\begin{proof} Observe that $M_n(\mathcal{B}_0)=M_n(\mathcal{B})_0$
for any integer $n\ge 2$ under the notations of
Definition~\ref{dfn:left_es_id}. Consequently, the assertion holds.
\end{proof}

\begin{teo}
  Let $A$ be a $C^*$-algebra and $\mathcal{B}$ be a normed
  algebra with involution containing $A$ as a left essential
  ideal. Then the following conditions are equivalent:
  \begin{enumerate}
    \item The algebra $\mathcal{B}$ is Banach with respect
    to the induced by the essential extension of $A$
    norm~{\rm(\ref{eq:str_es_ext_banalg})}.
    \item The algebra $M_n(\mathcal{B})$ is Banach with respect
    to the induced by the essential extension of $M_n(A)$
    norm~{\rm(\ref{eq:str_es_ext_banalg})} for any integer $n\ge 2$.
  \end{enumerate}
\end{teo}

\begin{proof} We will denote by $\|\cdot\|$  the
norm~(\ref{eq:str_es_ext_banalg}) for elements either from
$\mathcal{B}$ or from $M_n(\mathcal{B})$. For the convenience of
the reader let us remind the following well known inequalities
(cf.~\cite[Remark 3.4.1]{Murphy})
\begin{gather}
    \|a_{i,j}\|\le \|a\|\qquad (i,j=1,\dots,n)\label{eq:matr_est1}\\
      \|a\|\le\sum_{i,j=1}^n \|a_{i,j}\|\label{eq:matr_est2}
\end{gather}
which hold for any $a=(a_{i,j})\in M_n(A)$.

To begin with, let us prove that (i) implies (ii). For
$b=(b_{i,j})$ from $M_n(\mathcal{B})$ we have the following
estimates:
\begin{eqnarray}\label{eq:est_matr_str_es1}
    \|b\|
    & \ge &  \sup\left\{ \left\|\sum_{k=1}^n b_{i,k}a_{k,r}\right\| :
             \|(a_{i,j})\|\le 1, (a_{i,j})\in M_n(A)\right\} \\
    & \ge &  \sup\{\|b_{i,j}a_{j,r}\| : \|a_{j,r}\|\le 1, a_{j,r}\in
             A\} \nonumber \\
    & = &    \|b_{i,j}\|\nonumber
\end{eqnarray}
for all $1\le i,j\le n$ by (\ref{eq:matr_est1}).
Therefore, any Cauchy sequence $\{ b^{(N)}=(b^{(N)}_{i,j}): N \in
\mathbb{N} \}$ from $M_n(\mathcal{B})$ engenders Cauchy
sequences $\{ b^{(N)}_{i,j} \}$ from $\mathcal{B}$ for any pair $(i,j)$
with $1\le i,j\le n$. Let us denote the limits of the sequences
$\{ b^{(N)}_{i,j}: N \in \mathbb{N} \}$ in $\mathcal{B}$ by $b_{i,j}$
for any pair $(i,j)$, i.e.
\begin{gather*}
    \lim_{N \to \infty} \sup\left\{\left\|\left(b_{i,j}^{(N)}-
    b_{i,j}\right)a\right\| : \|a\|\le 1, a\in A\right\}=0 \, ,
\end{gather*}
and let $b=(b_{i,j})$ denote a corresponding matrix from
$M_n(\mathcal{B})$. Then, taking into
consideration~(\ref{eq:matr_est2}), we deduce
\begin{eqnarray}
\lefteqn{
    \left\|b-b^{(N)}\right\|\le \sup\left\{\sum_{i,j=1}^n
    \left\|\sum_{k=1}^n\left(b_{i,k}-b_{i,k}^{(N)}\right) a_{k,j}\right\| :
    \|(a_{k,j})\|\le 1, (a_{k,j})\in M_n(A)\right\}\le} \nonumber \\
& \le &
\sum_{i,j=1}^n\sum_{k=1}^n\sup\left\{\left\|\left(b_{i,k}-b_{i,k}^{(N)}
        \right)a_{k,j}\right\| : \|(a_{k,j})\|\le 1, (a_{k,j})\in M_n(A)\right\}
        \label{eq:est_matr_str_es2}\\
& = &   \sum_{i,j=1}^n\sum_{k=1}^n\sup\left\{\left\|\left(b_{i,k}-b_{i,k}^{(N)}
        \right)a_{k,j}\right\| : \|a_{k,j}\|\le 1, a_{k,j}\in A\right\} \nonumber \\
& = &   \sum_{i,k=1}^n n\|b_{i,k}-b_{i,k}^{(N)}\| \, . \nonumber
\end{eqnarray}
Therefore the sequence $\{ b^{(N)} \}$ converges to $b$ with respect
to the norm and, hence, the space $M_n(\mathcal{B})$ is complete.

Now we have to verify that (ii) implies (i). Let us consider
any Cauchy sequence $\{ b^{(N)}: N \in \mathbb N \}$ of $\mathcal{B}$
and define a corresponding sequence of matrices $\{ \widetilde{b}^{(N)}\}$
of $M_n(\mathcal{B})$ where the element at position $(1,1)$ of the
respectively derived matrix equals to $b^{(N)}$ and all the other elements
of the matrices are equal to zero. Then
inequality~(\ref{eq:est_matr_str_es2}) is the reason why
\begin{gather*}
    \|\widetilde{b}^{(N)}\|\le n\|b^{(N)}\|
\end{gather*}
for any $N \in \mathbb N$. Denote the limit of the sequence $\{
\widetilde{b}^{(N)}\}$ in $M_n(\mathcal{B})$ by
$\widetilde{b}=(\widetilde{b}_{i,j})$. Immediately
(\ref{eq:est_matr_str_es1}) implies that, firstly,
$\widetilde{b}_{i,j}=0$ if $(i,j)\neq (1,1)$ and, secondly, the
sequence $\{ b^{(N)} \}$ converges to $\widetilde{b}_{1,1}$.
\end{proof}

Finally, considering the particular case of maximal left strict
essential extensions of matrix algebras $M_n(A)$ for
$C^*$-algebras $A$ one obtains the identification
$LM(M_n(A))\simeq M_n(LM(A))$ for any integer $n \ge 1$. The
equality may be verified using the strict topology approach to
left multipliers of $C^*$-algebras (see~\cite{Wegge-Olsen} for
details).

\section{Left strict topology and left multipliers}

Essential left extensions of $C^*$-algebras are strongly
interrelated with some kind of topological closures of the
embedded copy of the extended $C^*$-algebra, where these
left strict topologies are generated by certain sets of
semi-norms. We are going to look for analogous sets of semi-norms
for essential extensions of $C^*$-algebras of bounded $C^*$-linear
operators on Hilbert $C^*$-modules and for essential extensions
of Hilbert $C^*$-modules.

Before we discuss an approach to the left multipliers of Hilbert
$C^*$-modules which is connected with the notion of a certain left
strict topology let us introduce the analogue of the one for the
left $(E,C,\rho)$-multipliers of a $C^*$-algebra $A$
(see~Definition~\ref{dfn:left_mult}).

\begin{dfn}\rm
Let $(E,C,\rho)$ be an admissible for $A$ triple. The left strict
topology on ${\rm End}_C(E)$ is defined
by the family of semi-norms
\begin{gather}\label{eq:ls_top_alg}
    \{\nu_a\}_{a\in A},\quad\text{where}\quad
    \nu_a(T)=\|T\rho(a)\|,\, T\in {\rm End}_C(E).
\end{gather}
We will denote this topology by $l.s.$
\end{dfn}

\begin{prop}
   Let $(E,C,\rho)$ be an admissible for a $C^*$-algebra $A$
   triple and $\mathcal{B}$ be a Banach subalgebra of ${\rm
   End}_C(E)$ containing $\rho(A)$ as a left ideal. The following
   conditions are equivalent:
   \begin{enumerate}
      \item the left strict topology on $\mathcal{B}$ is Hausdorff;
      \item $\rho(A)$ is an essential ideal of $\mathcal{B}$.
   \end{enumerate}
   In particular, the set of left $(E,C,\rho)$-multipliers of $A$
   equipped with the left strict topology is a Hausdorff space.
\end{prop}

\begin{proof} The second condition of
Definition~\ref{dfn:left_es_id} is, obviously, equivalent to the
requirement that the system~(\ref{eq:ls_top_alg}) of semi-norms
separates points of $\mathcal{B}$.
\end{proof}

\begin{prop}\label{prop:A_cl_l.s.t}
   The set of all left $(E,C,\rho)$-multipliers $LM_{(E,C,\rho)}(A)$
   of $A$ is a closed space with respect to the left strict topology.
\end{prop}

\begin{proof}
Let us take into consideration any net $\{T_\alpha\}$ from
$LM_{(E,C,\rho)}(A)$ converging to $T\in {\rm End}_C(E)$ with
respect to the left strict topology. It means that the net
$\{T_\alpha\rho(a)\}$ from $\rho(A)$ converges to $T\rho(a)$ with
respect to norm for any $a\in A$. Therefore $T\rho(a)$ belongs to
$\rho(A)$ for any $a \in A$ and, hence, $T$ belongs to
$LM_{(E,C,\rho)}(A)$.
\end{proof}

\begin{prop}
   The set of all left $(E,C,\rho)$-multipliers $LM_{(E,C,\rho)}
   (A)$ of $A$ coincides with the closure
   $\overline{\rho(A)}^{\, l.s.}$ of $\rho(A)$ inside ${\rm End}_C(E)$
   with respect to the left strict topology.
\end{prop}

\begin{proof}
Because of  Proposition~\ref{prop:A_cl_l.s.t} it remains to check
that $LM_{(E,C,\rho)}(A)\subset \overline{\rho(A)}^{\, l.s.}$.
Take an approximative identity $\{e_\alpha\}$ in $\rho(A)$. Then
for any $T\in LM_{(E,C,\rho)}(A)$ the net $\{Te_\alpha\} \in
\rho(A)$ converges to $T$ with respect to the left strict topology.
Indeed
\begin{gather*}
   \lim_\alpha\|Te_\alpha\rho(a)-T\rho(a)\|\le
   \lim_\alpha\|T\|\|e_\alpha\rho(a)-\rho(a)\|=0
\end{gather*}
for any $a\in A$.
\end{proof}

Now let us consider a Banach extension $(W,\mathcal{B},\Phi)$ of a
Hilbert $A$-module $V$. Then $V$ is isomorphic to $\Phi (V)$ and
${\rm Hom}_A (A,\Phi (V))$ is a Banach $LM(A)$-submodule of ${\rm
Hom}(A,W)$ with respect to the action~(\ref{eq:LM_action}).

\begin{dfn}\rm
The left strict topology on ${\rm Hom}_A(A,W)$ is defined by the
family of semi-norms
\begin{gather}\label{eq:ls_top_mod}
    \{\nu_a\}_{a\in A},\quad\text{where}\quad
    \nu_a(S)=\|Sa\|,\, S\in {\rm Hom}_A(A,W).
\end{gather}
Here we understand $A$ canonically embedded into $LM(A)$, and
$Sa$ means the result of the action~(\ref{eq:LM_action}).
We will denote this topology by $l.s.$
\end{dfn}

\begin{dfn}\rm\label{dfn:gen_left_mult_mod}
Let $(W,\mathcal{B},\Phi)$ be a Banach extension of a Hilbert
$A$-module $V$. The left
$(W,\mathcal{B},\Phi)$-multipliers of $V$ are defined as
\begin{gather*}
    LM_{(W,\mathcal{B},\Phi)}(V)={\rm Hom}_A(A,\Phi(V)) \, .
\end{gather*}
\end{dfn}

It is clear that left $(W,\mathcal{B},\Phi)$-multipliers of $V$
are isomorphic for all Banach extensions $(W,\mathcal{B},\Phi)$ of
$V$. Beside this, the previous Definition~\ref{dfn:left_mult_mod}
of left multipliers $LM(V)$ of $V$ is a special case of
Definition~\ref{dfn:gen_left_mult_mod} corresponding to the
identical Banach extension discussed as Example~\ref{ex:id_Ban_ext}.

In the sequel we will need the generalization $\Gamma_\Phi :
\Phi(V)\rightarrow LM_{(W,\mathcal{B},\Phi)}(V)$ of the map
(\ref{eq:def_Gamma}) which will be defined in the following way
\begin{gather}\label{eq:def_Gamma_Phi}
    (\Gamma_\Phi (y))(a):=ya, \qquad y\in \Phi(V), a\in A \, .
\end{gather}
This map is an isometric $A$-linear map.

\begin{prop}
   Let $(W,\mathcal{B},\Phi)$ be a Banach extension of a Hilbert
   $A$-module $V$. The following conditions are equivalent:
   \begin{enumerate}
      \item the left strict topology on ${\rm Hom}_A(A,W)$ is Hausdorff;
      \item the submodule $W_0$ of $W$ introduced
            in~{\rm (\ref{eq:es_ban_mod})} equals to zero.
   \end{enumerate}
   In particular, the set of all left
   $(W,\mathcal{B},\Phi)$-multipliers of $V$ equipped with the left
   strict topology is a Hausdorff space.
\end{prop}

\begin{proof}
Obviously, $W_0$ equals zero if and only if the system
(\ref{eq:ls_top_mod}) of semi-norms separates points of
${\rm Hom}_A(A,W)$.
\end{proof}

\begin{teo}\label{prop:LM(V)_cl_l.s.t}
   The set of all left $(W,\mathcal{B},\Phi)$-multipliers
   $LM_{(W,\mathcal{B},\Phi)}(V)$ of $V$ is a closed space with
   respect to the left strict topology.
\end{teo}

\begin{proof}
Consider any net $\{S_\alpha\} \in LM_{(W,\mathcal{B},\Phi)}(V)$
converging to $S \in {\rm Hom}_A(A,W)$ with respect to the left strict
topology. This means that the net $\{S_\alpha a\}$ from ${\rm Hom}_A
(A,\Phi(V))$ converges to $S a$ with respect to the norm for any
$a\in A$. Therefore, $Sa$ belongs to ${\rm Hom}_A(A,\Phi(V))$ for any
$a \in A$ and, consequently, $(Sa)(b)=S(ab)$ belongs to $\Phi(V)$
for any $a,b\in A$. The latter implies that the image of $S$ belongs
to $\Phi(V)$, i.e. $S$ is an element of $LM_{(W,\mathcal{B},\Phi)}(V)$.
\end{proof}

\begin{teo}
   The set of all left $(W,\mathcal{B},\Phi)$-multipliers
   $LM_{(W,\mathcal{B},\Phi)}(V)$ of $V$ coincides with the closure
   $\overline{\Gamma_\Phi(\Phi(V))}^{\, l.s.}$ of the image
   of the map {\rm (\ref{eq:def_Gamma_Phi})} inside ${\rm
   Hom}_A(A,W)$ with respect to left strict topology.
\end{teo}

\begin{proof}
Because of  Theorem~\ref{prop:LM(V)_cl_l.s.t} it
remains to check that $LM_{(W,\mathcal{B},\Phi)}(V)\subseteq
\overline{\Gamma_\Phi(\Phi(V))}^{\, l.s.}$. Consider an approximative
identity $\{e_\alpha\}$ for $A$. Then for any $S \in
LM_{(W,\mathcal{B},\Phi)}(V)$ the net $\{\Gamma_\Phi(S(e_\alpha))\}$
converges with respect to the left strict topology to $S$. Indeed,
\begin{eqnarray*}
   \lim_\alpha\|(\Gamma_\Phi(S(e_\alpha))a)(b)-(Sa)(b)\|
     & = & \lim_\alpha\|(\Gamma_\Phi(S(e_\alpha))(ab)-S(ab)\|  \\
     & = &  \lim_\alpha\|S(e_\alpha)ab-S(ab)\|  \\
     & = & \lim_\alpha\|S(e_\alpha ab)-S(ab)\|  \\
     & \le & \lim_\alpha\|S\|\|b\|\|e_\alpha a-a\|  \\
     & = & 0
\end{eqnarray*}
for any $a,b\in A$.
\end{proof}

As a summary, in the present paper we have extended the results of
D.~Baki\'c and B.~Gulja\v{s} from~\cite{Bakic} about multipliers
of Hilbert $C^*$-modules to the case of left multipliers of
Hilbert $C^*$-modules. The well known facts about the left
multiplier algebra of a $C^*$-algebra are particular cases of our
results in case a $C^*$-algebra is considered as a Hilbert
$C^*$-module over itself. In forthcoming research it would be
interesting to investigate an analogue of quasi-multipliers of
$C^*$-algebras for Hilbert $C^*$-modules.

{\bf Acknowledgement:} The presented work is a part of the
research project ''$K$-theory, $C^*$-algebras, and Index Theory''
of Deutsche Forschungsgemeinschaft (DFG). The authors are grateful
to DFG for the support. The research was fulfilled during the
visit of the second author to the Leipzig University of Applied
Sciences (HTWK), and he appreciates its hospitality a lot.

\end{document}